\documentclass[11pt]{article}

\usepackage[linesnumbered,ruled]{algorithm2e}

\usepackage{jhtitle}
\usepackage{geometry}
\usepackage{amsbsy,amsmath,amsthm,amssymb,graphicx}
\usepackage{natbib}
\usepackage{hyperref}
\usepackage{enumitem}
\usepackage{pstricks}
\usepackage{subcaption}
\usepackage{tcolorbox}

\newcommand{\pcite}[1]{\citeauthor{#1}'s \citeyearpar{#1}}

\def\baro{\vskip  .2truecm\hfill \hrule height.5pt \vskip  .2truecm}
\def\barba{\vskip -.1truecm\hfill \hrule height.5pt \vskip .4truecm}

\newtheorem{theorem}{Theorem}

\newtheorem{remark}[theorem]{Remark}
\newtheorem{proposition}[theorem]{Proposition}

\newcommand{\X}{{\mathsf{X}}}
\newcommand{\Y}{{\mathsf{Y}}}
\newcommand{\Z}{{\mathsf{Z}}}

\newcommand{\by}{{\boldsymbol{y}}}

\newcommand{\bth}{{\boldsymbol{\theta}}}

\newcommand{\norm}[1]{\lVert#1\rVert}
\newcommand{\bnorm}[1]{\big\lVert#1\big\rVert}

\newcommand{\df}{\mathrm{d}}

\title{\bf On the convergence rate of the ``out-of-order'' \\ block
  Gibbs sampler}

\author{Zhumengmeng Jin and James P. Hobert \\ Department of Statistics
	\\ University of Florida} 

\date{October 2021}
\keywords{Geometric ergodicity, geometric rate of convergence, Markov
  chain, Markov chain Monte Carlo, mixing time, total variation distance}

\begin{document}
\maketitle

\begin{abstract}
  It is shown that a seemingly harmless reordering of the steps in a
  block Gibbs sampler can actually invalidate the algorithm.  In
  particular, the Markov chain that is simulated by the
  ``out-of-order'' block Gibbs sampler does not have the correct
  invariant probability distribution.  However, despite having the
  wrong invariant distribution, the Markov chain converges at the same
  rate as the original block Gibbs Markov chain.  More specifically,
  it is shown that either both Markov chains are geometrically ergodic
  (with the same geometric rate of convergence), or neither one is.
  These results are important from a practical standpoint because the
  (invalid) out-of-order algorithm may be easier to analyze than the
  (valid) block Gibbs sampler (see, e.g., \cite{yang:rose:2019}).
\end{abstract}

\section{Introduction}
\label{sec:intro}

Let $\Pi(\df x,\df y,\df z)$ be a joint probability distribution having support
$\X \times \Y \times \Z$, and suppose we wish to construct a Gibbs
sampler for this distribution.  (Additional mathematical rigor will be
introduced in Section~\ref{sec:main}.)  The simplest version of the
Gibbs sampler is that which cycles through the so-called full
conditional distributions, $\Pi_{X|YZ}(\df x|y,z)$, $\Pi_{Y|XZ}(\df
y|x,z)$, and $\Pi_{Z|XY}(\df z|x,y)$.  Let $\Gamma =
\{(X_n,Y_n,Z_n)\}_{n=0}^\infty$ denote the corresponding Markov chain.
If the current state of the chain $\Gamma$ is $(X_n,Y_n,Z_n) =
(x,y,z)$, then we move to the next state, $(X_{n+1},Y_{n+1},Z_{n+1})$,
using the following well-known three step procedure.

\newpage

\baro \vspace*{2mm}
\noindent {\rm Iteration $n+1$ of the Gibbs sampler:}
\begin{enumerate}
\item Draw $X_{n+1} \sim \Pi_{X|YZ}(\cdot|y,z)$, and call the observed value $x'$.
\item Draw $Y_{n+1} \sim \Pi_{Y|XZ}(\cdot|x',z)$, and call the observed value $y'$.
\item Draw $Z_{n+1} \sim \Pi_{Z|XY}(\cdot|x',y')$.
\end{enumerate}
\vspace*{-3.5mm}
\barba 

Now suppose that we have the ability to draw from the distribution
$\Pi_{X|Z}(\df x|z)$.  This leads to a \textit{sequential} method of
making draws from $\Pi_{XY|Z}(\df x,\df y|z)$.
Indeed, we first draw from $\Pi_{X|Z}(\df x|z)$, and then from
$\Pi_{Y|XZ}(\df y|x,z)$.  Hence, we can consider an alternative
algorithm, the so-called block Gibbs sampler, in which $X$ and $Y$ are
sampled \textit{jointly} given $Z$.  It is well established that
sampling a set of variables jointly in this way often leads to a more
efficient algorithm.  Denote the block Gibbs Markov chain by
$\tilde{\Gamma} = \{\big( (\tilde{X}_n,\tilde{Y}_n),\tilde{Z}_n
\big)\}_{n=0}^\infty$.  If the current state is $\big(
(\tilde{X}_n,\tilde{Y}_n),\tilde{Z}_n \big) = ((x,y),z)$, then we move
to the next state $\big(
(\tilde{X}_{n+1},\tilde{Y}_{n+1}),\tilde{Z}_{n+1} \big)$ using the
following three step procedure.

\baro \vspace*{2mm}
\noindent {\rm Iteration $n+1$ of the block Gibbs sampler:}
\begin{enumerate}
\item Draw $\tilde{X}_{n+1} \sim \Pi_{X|Z}(\cdot|z)$, and call the observed value $x'$.
\item Draw $\tilde{Y}_{n+1} \sim \Pi_{Y|XZ}(\cdot|x',z)$, and call the observed value $y'$.
\item Draw $\tilde{Z}_{n+1} \sim
  \Pi_{Z|XY}(\cdot|x', y')$.
\end{enumerate}
\vspace*{-3.5mm}
\barba

The target distribution, $\Pi(\df x,\df y,\df z)$, is the invariant probability
distribution for both $\Gamma$ and $\tilde{\Gamma}$.  (The Markov
chains considered in this section are assumed to satisfy standard
regularity conditions - spelled out in Section~\ref{sec:main} - that
imply, among other things, a unique invariant distribution.)
Therefore, the Gibbs sampler and the block Gibbs sampler are both
valid Markov chain Monte Carlo (MCMC) algorithms for exploring
$\Pi(\df x,\df y,\df z)$.  There is, however, a very important difference
between these two algorithms regarding the ordering of the steps.  The
three steps of the regular Gibbs sampler are interchangeable in the
sense that they can be reordered without affecting the validity of the
algorithm.  In other words, instead of using the order $X \rightarrow
Y \rightarrow Z$, any of the six possible orderings could be used, and
all will result in a Markov chain with invariant distribution $\Pi(\df
x,\df y,\df z)$.  This is not the case, however, with the block Gibbs
sampler.  Indeed, the first two steps of the block Gibbs sampler must
be done in order because the same value of $Z$ must be conditioned on
in both.  So, we could actually move the third step to the top, and
still have a valid algorithm, but any other change in the order of the
steps will invalidate the algorithm.  To be more specific, let
$\Gamma^* = \{(Y^*_n,Z^*_n,X^*_n)\}_{n=0}^\infty$ denote the Markov
chain associated with what we call the ``out-of-order'' block Gibbs
sampler.  If the current state is $(Y^*_n,Z^*_n,X^*_n) = (y,z,x)$,
then we move to the next state $(Y^*_{n+1},Z^*_{n+1},X^*_{n+1})$ using
the following three step procedure.

\baro \vspace*{2mm}
\noindent {\rm Iteration $n+1$ of the out-of-order block Gibbs
  sampler:}
\begin{enumerate}
\item Draw $Y^*_{n+1} \sim \Pi_{Y|XZ}(\cdot|x,z)$, and call the observed value $y^\ast$.
\item Draw $Z^*_{n+1} \sim \Pi_{Z|XY}(\cdot|x, y^\ast)$, and call the observed value $z^\ast$.
\item Draw $X^*_{n+1} \sim \Pi_{X|Z}(\cdot|z^\ast)$.
\end{enumerate}
\vspace*{-3.5mm}
\barba 

At first glance, this reordering may seem perfectly harmless, but, in
reality, it invalidates the algorithm.  Indeed, $\Gamma^*$ does not
have the correct invariant distribution.  Again, the reason is that
the two steps in the block Gibbs sampler that together yield a draw
from $\Pi_{XY|Z}(\df x, \df y|z)$ have been separated, so we are no
longer simulating from $\Pi_{XY|Z}(\df x,\df y|z)$.  To see that the
invariant distribution has changed, first note that the Markov
transition kernel (Mtk) of $\Gamma^*$ is given by
\[
K^{\ast} \big( (y,z,x), (\df y',\df z',\df x') \big) =
\Pi_{X|Z} (\df x'|z') \,
\Pi_{Z|XY}(\df z'|x,y') \,
\Pi_{Y|XZ}(\df y'|x,z) \;.
\]

\begin{remark}
  Note that $K^{\ast}\big( (y,z,x), (\df y',\df z',\df x') \big)$ does
  not actually depend on $y$.  In the sequel, we drop such variables
  from the notation when it's convenient.
\end{remark}

\noindent Now defining
$\Pi^{\ast} (\df x,\df y,\df z) = \Pi_{X|Z}(\df x|z) \Pi_{YZ}(\df y,\df z)$,
we have
\begin{align*}
  & \int_{\mathsf{X}} \int_{\mathsf{Y}} \int_{\mathsf{Z}}
  K^{\ast}\big( (y,z,x), (\df y',\df z',\df x') \big) \, \Pi^*(\df x,\df y,\df z) \\
  = & \Pi_{X|Z}(\df x'|z') \int_{\mathsf{X}} \int_{\mathsf{Y}} \int_{\mathsf{Z}}
    \Pi_{Z|XY}(\df z'|x,y') \,
    \Pi_{Y|XZ}(\df y'|x,z) \,
    \Pi_{X|Z}(\df x|z) \,
    \Pi_{YZ}(\df y,\df z) \\
  = &  \Pi_{X|Z}(\df x'|z') \int_{\mathsf{X}} \int_{\mathsf{Z}}
       \Pi_{Z|XY}(\df z'|x,y') \,
       \Pi_{Y|XZ}(\df y'|x,z) \,
       \Pi_{XZ}(\df x,\df z)  \\
  = & \Pi_{X|Z}(\df x'|z') \int_{\mathsf{X}}
      \Pi_{Z|XY}(\df z'|x,y') \,
      \Pi_{XY}(\df x,\df y') \\
  = & \Pi_{X|Z}(\df x'|z') \,  \Pi_{YZ}(\df y',\df z') \\
  = & \Pi^*(\df x',\df y',\df z')    \;.
\end{align*}
Hence, the invariant distribution of $\Gamma^*$ is $\Pi^*(\df x,\df
y,\df z)$, not the target distribution $\Pi(\df x,\df y,\df z)$.  A
little thought reveals that basing an MCMC algorithm on $\Gamma^*$ is
effectively the same as simulating the correct chain, $\tilde{\Gamma}
= \{\big( (\tilde{X}_n,\tilde{Y}_n),\tilde{Z}_n \big)\}_{n=0}^\infty$,
but then basing inferences on the \textit{shifted} version $\{ (
\tilde{Y}_n, \tilde{Z}_n, \tilde{X}_{n+1} )\}_{n=0}^\infty$.  Since
$\Pi$ and $\Pi^*$ share many of the same marginal and conditional
distributions, some (MCMC) estimators based on $\Gamma^*$ will be
valid, but others will not.  For example, since (in general)
$\Pi^*_{XY}(\df x,\df y) = \int_{\Z} \Pi_{X|Z}(\df x|z) \Pi_{YZ}(\df
y,\df z) \ne \Pi_{XY}(\df x,\df y)$, the $\Gamma^*$-based estimator of
the expectation of $g(X,Y)$ with respect to $\Pi$, that is, $n^{-1}
\sum_{i=0}^{n-1} g(X_i^*,Y_i^*)$, will be inconsistent (assuming that
$g(X,Y)$ actually depends on \textit{both} $X$ and $Y$).

We now explain how our work on the out-of-order block Gibbs sampler
was motivated by an analysis in \citet{yang:rose:2019}.  Consider the
following simple random effects model:
\begin{equation}
  \label{eq:brem}
  y_i = \theta_i + e_i \,,
\end{equation}
for $i = 1,\dots,m$, where the components of $\bth =
(\theta_1,\dots,\theta_m)^T$ are iid $\mbox{N}(\mu,A)$, the components
of $\boldsymbol{e} = (e_1,\dots,e_m)^T$ are iid $\mbox{N}(0,V)$, and
$\bth$ and $\boldsymbol{e}$ are independent.  The error variance, $V$,
is assumed known.  The basic idea is that we will observe the $y_i$s,
and then attempt to make inferences about the unknown parameters, $A$
and $\mu$.  Consider a Bayesian statistical model in which $A$ and
$\mu$ are taken to be \textit{a priori} independent with
\begin{equation}
  \label{eq:bp}
  \pi(\mu) \propto 1 \;\;\; \mbox{and} \;\;\; A \sim \mbox{IG}(a, b)
  \;,
\end{equation}
where $a,b>0$,
and we say $W \sim \mbox{IG}(a, b)$ if its density is proportional to
$w^{-a-1} e^{-b/w} I_{(0,\infty)}(w)$.
Denote the
resulting posterior distribution as $\Pi(\df A,\df \mu,\df \bth \, |
\, \by)$, where $\by = (y_1,\dots,y_m)^T$ represents the observed
data.

\citet{rose:1996} derived and analyzed a block Gibbs sampler for
$\Pi(\df A,\df \mu,\df \bth \, | \, \by)$.  In terms of the general
block Gibbs sampler described above, $(A,\mu)$ plays the role of
$(X,Y)$, and $\bth$ plays the role of $Z$.  So $\X = (0,\infty)$, $\Y
= \mathbb{R}$ and $\Z = \mathbb{R}^m$.  Denote the Markov chain by
$\tilde{\Lambda} = \{\big((A_n,\mu_n),\bth_n \big)\}_{n=0}^\infty$.
If the current state is $\big( (A_n,\mu_n),\bth_n \big) = \big(
(A,\mu), \bth \big)$, then we move to the next state $\big(
(A_{n+1},\mu_{n+1}), \bth_{n+1} \big)$ using the following three step
procedure.

\baro \vspace*{2mm}
\noindent {\rm Iteration $n+1$ of Rosenthal's block Gibbs sampler:}
\begin{enumerate}
\item Draw $A_{n+1} \sim \mbox{IG} \Big( a + \frac{m-1}{2}, b +
  \frac{1}{2} \sum_{i=1}^m (\theta_i - \bar{\theta})^2 \Big)$.
\item Draw $\mu_{n+1} \sim \mbox{N} \Big( \bar{\theta},
  \frac{A_{n+1}}{m} \Big)$.
\item For $i=1,2,\dots,m$, draw the $i$th component of $\bth_{n+1}$
  from $\mbox{N} \Big( \frac{V \mu_{n+1} + A_{n+1} y_i}{A_{n+1}+V},
  \frac{A_{n+1} V}{A_{n+1}+V} \Big)$.
\end{enumerate}
\vspace*{-3.5mm}
\barba 

The first two steps of this algorithm constitute a joint draw from the
distribution of $(A,\mu)$ given $(\bth,\by)$, which we write as
$\Pi_1(\df A, \df \mu \, | \, \bth, \by)$.  In particular, if we
factor $\Pi_1$ as follows
\[
\Pi_1(\df A, \df \mu \, | \, \bth, \by) = \Pi_{21}(\df A \, | \, \bth,
\by) \Pi_{22}(\df \mu \, |\, A, \bth, \by) \;,
\]
then in the first step we draw $A_{n+1} \sim \Pi_{21}(\cdot \, | \,
\bth, \by)$, and in the second step we draw $\mu_{n+1} \sim
\Pi_{22}(\cdot \, | \, A_{n+1}, \bth, \by)$.  In the posterior
distribution, the components of $\bth$ are conditionally independent
given $(\mu, A, \by)$, which is why the third step of the algorithm
consists of $m$ univariate draws.

Recently, \citet{yang:rose:2019} set out to perform a more nuanced
analysis of \pcite{rose:1996} block Gibbs sampler, but ended up
analyzing the out-of-order version of \pcite{rose:1996} algorithm
instead.  Let $\Lambda^* = \{(\mu^*_n,\bth^*_n,A^*_n)\}_{n=0}^\infty$
denote their Markov chain.  If the current state is
$(\mu^*_n,\bth^*_n,A^*_n) = (\mu,\bth,A)$, then we move to the next
state $(\mu^*_{n+1},\bth^*_{n+1},A^*_{n+1})$ using the following three
step procedure.

\baro \vspace*{2mm}
\noindent {\rm Iteration $n+1$ of Yang \& Rosenthal's algorithm:}
\begin{enumerate}
\item Draw $\mu^*_{n+1} \sim \mbox{N} \big( \bar{\theta}, \frac{A}{m}
  \big)$.
\item For $i=1,2,\dots,m$, draw the $i$th component of $\bth^*_{n+1}$
  from $\mbox{N} \Big( \frac{V \mu^*_{n+1} + A y_i}{A+V}, \frac{A
    V}{A+V} \Big)$.
\item Draw $A^*_{n+1} \sim \mbox{IG} \Big( a + \frac{m-1}{2}, b +
  \frac{1}{2} \sum_{i=1}^m (\theta^*_{n+1,i} - \bar{\theta}_{n+1}^*)^2
  \Big)$.
\end{enumerate}
\vspace*{-3.5mm}
\barba

It appears that \citet{yang:rose:2019} changed the order of the steps
in \pcite{rose:1996} block Gibbs sampler because the resulting Markov
chain was easier to analyze.  However, they were apparently unaware
that the modified algorithm they analyzed is not a valid MCMC
algorithm for the target posterior distribution $\Pi(\df A,\df \mu,\df
\bth \, | \, \by)$.  Thus, without supplementary results relating
$\tilde{\Lambda}$ and $\Lambda^*$, the convergence results developed
by \citet{yang:rose:2019} are not entirely germane to the exploration
of $\Pi(\df A,\df \mu,\df \bth \, | \, \by)$.

In this note, we develop general results relating the convergence
behaviors of the (valid) block Gibbs Markov chain, $\tilde{\Gamma}$,
and the (invalid) out-of-order block Gibbs Markov chain, $\Gamma^*$.
In particular, we develop inequalities that relate the total variation
distance to stationarity for one Markov chain to that of the other,
and these bounds imply that the two chains mix at essentially the same
rate.  (Keep in mind that these two Markov chains have different
invariant distributions.)  We also show that $\tilde{\Gamma}$ is
geometrically ergodic if and only if $\Gamma^*$ is geometrically
ergodic, and that their convergence rates are the same.  Our results
demonstrate that one can perform convergence analysis of the block
Gibbs sampler \textit{indirectly} by analyzing the out-of-order block
Gibbs sampler, which may be easier to handle, as was apparently the
case for \citet{yang:rose:2019}.

\section{Main Results}
\label{sec:main}

We begin with some general state space Markov chain theory.  Let $E$
be a set, and let ${\cal E}$ be a countably generated $\sigma$-algebra
of subsets of $E$.  Let $P: E \times {\cal E} \rightarrow [0,1]$ be a
Mtk, and assume that the corresponding Markov chain is irreducible,
aperiodic and positive Harris recurrent.  (See \citet{meyn:twee:2009}
for definitions.)  Let $\Pi$ denote the unique invariant probability
measure, and let $P^n$ denote the $n$-step Mtk (so, as usual, $P
\equiv P^1$).  If $\mu$ is a probability measure on $E$, then we use
$\mu P^n(\cdot)$ to denote the probability measure $\int_E
P^n(u,\cdot) \mu(du)$.  Following \citet{numm:1984} and
\citet{tier:1994}, we say that the Markov chain is
\textit{geometrically ergodic} if there exist a non-negative extended
real-valued function $M$ with $\int_E M(u) \, \Pi(du) < \infty$ and a
$\rho \in [0,1)$ such that
\begin{equation}
  \label{eq:ge}
  \norm{P^n(u,\cdot) - \Pi(\cdot)} \le M(u) \rho^n
\end{equation}
for all $n \in \mathbb{N}$ and all $u \in E$.  Here, $\norm{\cdot}$
denotes the usual total variation norm for signed measures.  We define
the \textit{geometric convergence rate} of the chain, $\rho_*$, to be
the infimum over $\rho \in [0,1]$ such that \eqref{eq:ge} holds for
some non-negative extended real-valued function $M$, integrable with
respect to $\Pi$.  Clearly, the chain is geometrically ergodic if and
only if $\rho_* < 1$.

\begin{remark}
  This definition of geometric ergodicity is slightly less standard
  than an alternative definition in which $M$ is not assumed to be
  integrable with respect to $\Pi$, but is assumed to be finite
  $\Pi$-almost everywhere.  While \pcite{robe:rose:1997} Proposition
  2.1 shows that the two definitions are actually equivalent, the
  geometric convergence rates under the two definitions are not
  necessarily equal.
\end{remark}

A standard result concerning total variation distance is as follows
\citep[see, e.g.,][Proposition 3]{robe:rose:2004}.  If $\mu$ and $\nu$
are probability measures on $E$, then
\[
\norm{\mu(\cdot) - \nu(\cdot)} = \sup_{f:E \rightarrow [0,1]} \bigg|
\int f d\mu - \int f d\nu \bigg| \;.
\]
This fact will be used repeatedly in our proofs.

Now let $K$ denote the Mtk of the block Gibbs Markov chain,
$\tilde{\Gamma} = \{\big( (\tilde{X}_n,\tilde{Y}_n),\tilde{Z}_n
\big)\}_{n=0}^\infty$, described in the Introduction.  So we have
\[
K \big( (x,y,z) , (\df x',\df y',\df z') \big) = \Pi_{Z|XY}(\df z'\, |
\, x',y' ) \, \Pi_{XY|Z}( \df x',\df y' \, | \, z) \;.
\]
(We assume that
all Markov chains considered in this section
satisfy
the basic regularity conditions laid out at the start of this
section.)
It is well known that the sequences $\{
(\tilde{X}_n,\tilde{Y}_n) \}_{n=0}^\infty$ and $\{ \tilde{Z}_n
\}_{n=0}^\infty$ are themselves (reversible) Markov chains with
invariant distributions given by $\Pi_{XY}(\df x,\df y)$ and
$\Pi_Z(\df z)$, respectively (see, e.g., \cite{liu:wong:kong:1994}).
Let $K_{XY}$ denote the Mtk of the $(X,Y)$-marginal chain, and define
$K_Z$ similarly.  For example,
\[
K_{XY} \big( (x,y), (\df x',\df y') \big)
= \int_{\Z} \Pi_{XY|Z}(\df x',\df y' \, |\, z) \,
\Pi_{Z|XY}(\df z \, | \, x,y) \;.
\]
As mentioned in the Introduction, there is one way to reorder the
steps in the block Gibbs sampler without invalidating the algorithm.
The corresponding Markov chain has Mtk given by
\[
K^\dagger\big( (z,x,y), (\df z',\df x',\df y') \big)
= \Pi_{XY|Z}(\df x',\df y' \,|\, z')
  \Pi_{Z|XY}(\df z'\,|\, x,y ) \;.
\]
All four of these Markov chains ($K$, $K^\dagger$, $K_{XY}$ and $K_Z$)
converge at the same rate \citep{robe:rose:2001,diac:khar:salo:2008}.
As in the Introduction, let $K^*$ denote the Mtk of the out-of-order
chain, so, again,
\[
K^* \big( (y,z,x), (\df y',\df z',\df x') \big)
= \Pi_{X|Z} (\df x'|z') \,
  \Pi_{Z|XY}(\df z'|x,y') \,
  \Pi_{Y|XZ}(\df y'|x,z) \;.
\]
Here is our first result.

\begin{proposition}
  \label{prop:first}
  Fix $z \in \Z$ and let $\nu_z$ denote the probability measure of the
  random vector $(X,Z)$ such that $Z=z$ w.p. 1, and the conditional
  distribution of $X$ given $Z$ is $\Pi_{X|Z}(\cdot|z)$.  Then
\begin{equation}\label{eq:tvr1}
\norm{K^n\big((x,y,z),\cdot \big) - \Pi(\cdot)} \le
\norm{K_Z^{n-1}(z,\cdot) - \Pi_Z(\cdot)} \le \norm{\nu_z
  K^{*(n-2)}(\cdot) - \Pi^{\ast}(\cdot)} \;.
\end{equation}
  Fix $(x,z) \in \X \times \Z$ and let $\nu_{(x,z)}$ denote the
  probability measure of the random vector $(X,Y)$ such that $X=x$
  w.p. 1, and the conditional distribution of $Y$ given $X$ is
  $\Pi_{Y|XZ}(\cdot|x,z)$.  Then
\begin{equation}\label{eq:tvr2}
  \norm{K^{*n}\big( (y,z,x), \cdot) - \Pi^*(\cdot)} \le
  \norm{\nu_{(x,z)} K_{XY}^{n-1}(\cdot) - \Pi_{XY}(\cdot)} \le \norm{\nu_{(x,z)}
    K^{\dagger(n-1)}(\cdot) - \Pi(\cdot)} \;.
\end{equation}
\end{proposition}

\begin{proof}
We begin with \eqref{eq:tvr1}.  The first inequality follows
immediately from Lemma 2.4 of \citet{diac:khar:salo:2008}. Now, it is
straightforward to show that
\begin{align}
  \label{eq:1}
  & K^n_Z (z, \df z') \notag \\
= & \int_{\X} \int_{\Y} \int_{\Z} \int_{\X }\int_{\Y}
    \Pi_{Z|XY}(\df z'|x',y') \Pi_{Y|XZ}(\df y'|x',z'')
     K^{*(n-1)}\big( (x'',z), (\df y'', \df z'', \df x') \big) \Pi_{X|Z}( \df x''|z)
      \notag \\
= & \int_{\Y} \int_{\Z} \int_{\X }\int_{\Y}
    \Pi_{Z|XY}(\df z'|x',y') \Pi_{Y|XZ}(\df y'|x',z'')
    \int_{\X}  K^{*(n-1)}\big( (x'',z), (\df y'', \df z'', \df x') \big) \Pi_{X|Z}( \df x''|z) \notag \\
= & \int_{\Y} \int_{\Z} \int_{\X }\int_{\Y}
    \Pi_{Z|XY}(\df z'|x',y') \Pi_{Y|XZ}(\df y'|x',z'')
    \nu_z K^{*(n-1)}(\df y'', \df z'', \df x')  \;. 
\end{align}
Also,
\begin{align*}
  \Pi_Z(\df z')
  & = \int_\X \int_\Y \int_\Z \int_\Y
    \Pi_{Z|XY}(\df z'|x',y')
    \Pi_{Y|XZ}(\df y'|x',z'')
    \Pi_{YZ}(\df y'', \df z'')
    \Pi_{X|Z}(\df x'|z'') \\
  & = \int_\X \int_\Y \int_\Z \int_\Y
    \Pi_{Z|XY}(\df z'|x',y')
    \Pi_{Y|XZ}(\df y'|x',z'')
    \Pi^{\ast}(\df x', \df y'', \df z'') \;.
\end{align*}
Hence for the second inequality in \eqref{eq:tvr1},
\begin{align*}
 & \bnorm{K^n_Z(z, \cdot) - \Pi_Z(\cdot)}  \notag \\
= & \sup_{0\le f\le 1} \bigg| \int_{\Z} f(z') \int_{\Y} \int_{\Z} \int_{\X} \int_{\Y}
     \Pi_{Z|XY}(\df z'|x',y') \Pi_{Y|XZ}(\df y'|x',z'')
     \nu_z K^{*(n-1)}(\df y'', \df z'', \df x')  \\
  & \hspace*{10mm} - \int_{\Z} f(z') \int_{\Y} \int_{\Z} \int_{\X} \int_{\Y}
     \Pi_{Z|XY}(\df z'|x',y') \Pi_{Y|XZ}(\df y'|x',z'')
     \Pi^{\ast}(\df x', \df y'', \df z'')
    \bigg|  \\
= & \sup_{0\le f\le 1} \bigg| \int_{\Z} \int_{\X} \int_{\Y}
     \bigg[\int_{\Y} \int_{\Z} f(z') \Pi_{Z|XY}(\df z'|x',y') \Pi_{Y|XZ}(\df y'|x',z'')\bigg] \nu_z K^{*(n-1)}(\df y'', \df z'', \df x') \\
  & \hspace*{10mm} - \int_{\Z} \int_{\X} \int_{\Y}
     \bigg[\int_{\Y} \int_{\Z} f(z')  \Pi_{Z|XY}(\df z'|x',y') \Pi_{Y|XZ}(\df y'|x',z'')\bigg]
     \Pi^{\ast}(\df x', \df y'', \df z'')
    \bigg|  \\
= & \sup_{0\le f\le 1} \bigg| \int_{\Z} \int_{\X} \int_{\Y}
     \bigg[ \int_{\Y} \int_{\Z} f(z') \Pi_{Z|XY}(\df z'|x',y') \Pi_{Y|XZ}(\df y'|x',z'') \bigg]  \notag \\
  & \hspace*{70mm} \bigg[ \nu_z K^{*(n-1)}(\df y'', \df z'', \df x')
    - \Pi^{\ast}(\df x', \df y'', \df z'')  \bigg]
    \bigg|  \\
\le & \sup_{0\le g\le 1} \bigg|   \int_{\Z} \int_{\X} \int_{\Y} g(x', y'', z'')
     \bigg[ \nu_z K^{*(n-1)}(\df y'', \df z'', \df x')
            - \Pi^{\ast}(\df x', \df y'', \df z'')  \bigg]
    \bigg|  \notag \\
= & \bnorm{\nu_z K^{*(n-1)}(\cdot) - \Pi^{\ast}(\cdot)}  \;. \notag
\end{align*}
Now for \eqref{eq:tvr2}.  It is straightforward to show that
\begin{align}
  & K^{*n}\big( (y,z,x), (\df y', \df z',\df x') \big) \notag \\
= & \int_{\Y} \int_{\X}
  \Pi_{X|Z}(\df x'|z')
  \Pi_{Z|XY}(\df z'|x'', y')
  K_{XY}^{n-1}\big( (x, y''), (\df x'', \df y') \big)
  \Pi_{Y|XZ}(\df y''|x,z)  \label{eq:2}\\
= & \int_{\X}
  \Pi_{X|Z}(\df x'|z')
  \Pi_{Z|XY}(\df z'|x'', y')
  \bigg[ \int_{\Y} K_{XY}^{n-1}\big( (x, y''), (\df x'', \df y') \big)
  \Pi_{Y|XZ}(\df y''|x,z) \bigg] \notag \\
= & \int_{\X}
  \Pi_{X|Z}(\df x'|z')
  \Pi_{Z|XY}(\df z'|x'', y')
  \nu_{(x,z)} K_{XY}^{n-1} (\df x'', \df y')  \;. \notag
\end{align}
Also,
\[
  \Pi^*(\df y',\df z',\df x')
  = \Pi_{X|Z}(\df x'|z') \Pi_{YZ}(\df y',\df z')
  = \int_{\X}
    \Pi_{X|Z}(\df x'|z') \Pi_{Z|XY}(\df z'|x'',y') \Pi_{XY}(\df x'',\df y')  \;.
\]
Hence for the first inequality of \eqref{eq:tvr2},
\begin{align*}
& \bnorm{K^{\ast n}\big((y,z,x), \cdot \big) - \Pi^{\ast}(\cdot)}  \\
= & \sup_{0\le f\le 1} \bigg| \int_{\X} \int_{\Y} \int_{\Z}  f(x', y', z')
    \int_{\X} \Pi_{X|Z}(\df x'|z') \Pi_{Z|XY}(\df z'|x'', y')
    \nu_{(x,z)} K_{XY}^{n-1} (\df x'', \df y')\\
  & \hspace*{10mm} - \int_{\X} \int_{\Y} \int_{\Z}  f(x', y', z')
      \int_{\X} \Pi_{X|Z}(\df x'|z') \Pi_{Z|XY}(\df z'|x'', y')
      \Pi_{XY}(\df x'',\df y')
    \bigg|\\
= & \sup_{0\le f\le 1} \bigg| \int_{\X} \int_{\Y}
    \bigg[\int_{\Z} \int_{\X} f(x', y', z')\Pi_{X|Z}(\df x'|z') \Pi_{Z|XY}(\df z'|x'', y') \bigg]
    \nu_{(x,z)} K_{XY}^{n-1} (\df x'', \df y')\\
  & \hspace*{10mm} - \int_{\X} \int_{\Y}
    \bigg[ \int_{\Z} \int_{\X} f(x', y', z') \Pi_{X|Z}(\df x'|z') \Pi_{Z|XY}(\df z'|x'', y') \bigg]
      \Pi_{XY}(\df x'',\df y')
    \bigg|\\
= & \sup_{0\le f\le 1} \bigg| \int_{\X} \int_{\Y}
    \bigg[\int_{\Z} \int_{\X} f(x', y', z')\Pi_{X|Z}(\df x'|z') \Pi_{Z|XY}(\df z'|x'', y') \bigg] \\
    & \hspace*{65mm} \bigg[\nu_{(x,z)} K_{XY}^{n-1} (\df x'', \df y')-\Pi_{XY}(\df x'',\df y') \bigg]
    \bigg|\\
\le & \sup_{0\le g\le 1} \bigg| \int_{\X} \int_{\Y}   g(x'', y')
    \bigg[
      \nu_{(x,z)} K_{XY}^{n-1} (\df x'', \df y')
    - \Pi_{XY}(\df x'',\df y')
    \bigg]  \bigg|\\
= & \bnorm{\nu_{(x,z)} K_{XY}^{n-1}(\cdot) - \Pi_{XY}(\cdot)} \;,
\end{align*}
The second inequality follows from the fact that
\[
\nu_{(x,z)} K_{XY}^n(\df x',\df y')
= \int_{\Z} \nu_{(x,z)}
            K^{\dagger n}\big( \df z',(\df x',\df y') \big) \;.
\]
\end{proof}

In light of Proposition~\ref{prop:first}, it should not be surprising
that the block Gibbs Markov chain is geometrically ergodic if and only
if the out-of-order block Gibbs Markov chain is geometrically ergodic,
and that they have the same rate of convergence.  This is the subject
of our next result.

\begin{proposition}
  \label{prop:second}
  There are only two possibilities:
  (i) $\tilde{\Gamma}$ and $\Gamma^*$ are both geometrically ergodic
  with the same geometric convergence rate, or
  (ii) neither chain is geometrically ergodic.
\end{proposition}

\begin{proof}
First, assume that $\Gamma^*$ is geometrically ergodic.
That is, for all $n \in \mathbb{N}$ and all $(y,z,x) \in \Y \times \Z \times \X$,
\[
\bnorm{K^{*n}((y,z,x),\cdot) - \Pi^*(\cdot)} \le M_1(x,z) \rho_1^n
\]
where $\rho_1 \in [0,1)$ and $\int_\X \int_\Y \int_\Z M_1(x,z)
  \Pi^*(\df x, \df y, \df z) < \infty$.  We now show that this implies
  the geometric ergodicity of the $Z$-marginal.  It's easy to see that
\begin{align*}
  & \Pi_Z(\df z') \\
= & \int_{\X} \int_{\X} \int_{\Y} \int_{\Z} \int_{\Y}
  \Pi_{Z|XY}(\df z'|x',y') \Pi_{Y|XZ}(\df y'|x',z'')
  \Pi_{YZ}(\df y'', \df z'') \Pi_{X|Z}(\df x'|z'')
  \Pi_{X|Z}(\df x''|z) \\
= & \int_{\X} \int_{\X} \int_{\Y} \int_{\Z} \int_{\Y}
  \Pi_{Z|XY}(\df z'|x',y') \Pi_{Y|XZ}(\df y'|x',z'')
  \Pi^\ast(\df x', \df y'', \df z'')
  \Pi_{X|Z}(\df x''|z) \;.
\end{align*}
Combining this with \eqref{eq:1}, we have
\begin{align*}
  & \bnorm{K^n_Z(z,\cdot) - \Pi_Z(\cdot)} \\
= & \sup_{0\le f \le 1}  \bigg| \int_{\Z} f(z')
    \int_{\X} \int_{\X} \int_{\Y} \int_{\Z} \int_{\Y}
    \Pi_{Z|XY}(\df z'|x',y') \Pi_{Y|XZ}(\df y'|x',z'') \\
  & \hspace*{60mm} K^{\ast (n-1)} \big( (x'',z), (\df y'', \df z'', \df x') \big)
    \Pi_{X|Z}(\df x''|z) \\
  & \hspace*{10mm} - \int_{\Z} f(z')
      \int_{\X} \int_{\X} \int_{\Y} \int_{\Z} \int_{\Y}
      \Pi_{Z|XY}(\df z'|x',y') \Pi_{Y|XZ}(\df y'|x',z'')
      \Pi^*(\df x', \df y'', \df z'')
      \Pi_{X|Z}(\df x''|z) \bigg|  \\
= & \sup_{0\le f \le 1}  \bigg|
    \int_{\X} \int_{\X} \int_{\Y} \int_{\Z}
    \bigg[\int_{\Y} \int_{\Z} f(z') \Pi_{Z|XY}(\df z'|x',y') \Pi_{Y|XZ}(\df y'|x',z'') \bigg] \\
  & \hspace*{60mm} K^{\ast (n-1)} \big( (x'',z), (\df y'', \df z'', \df x') \big)
    \Pi_{X|Z}(\df x''|z) \\
  & \hspace*{10mm} -
      \int_{\X} \int_{\X} \int_{\Y} \int_{\Z}
      \bigg[ \int_{\Y} \int_{\Z} f(z') \Pi_{Z|XY}(\df z'|x',y') \Pi_{Y|XZ}(\df y'|x',z'')
      \bigg]
      \Pi^*(\df x', \df y'', \df z'')
      \Pi_{X|Z}(\df x''|z) \bigg| \\
\le & \sup_{0\le g \le 1}  \bigg|
    \int_{\X} \int_{\X} \int_{\Y} \int_{\Z} g(x', y'', z'') \\
    & \hspace*{30mm} \bigg[ K^{\ast (n-1)} \big( (x'',z), (\df y'', \df z'', \df x') \big)
         - \Pi^*(\df x', \df y'', \df z'')
    \bigg] \Pi_{X|Z}(\df x''|z) \bigg| \\
\le & \int_{\X}
    \bigg\{ \sup_{0\le g \le 1} \bigg|
    \int_{\X} \int_{\Y} \int_{\Z} g(x', y'', z'') \\
    & \hspace*{30mm} \bigg[ K^{\ast (n-1)} \big( (x'',z), (\df y'', \df z'', \df x') \big)
         - \Pi^*(\df x', \df y'', \df z'')
    \bigg] \bigg| \bigg\}
    \Pi_{X|Z}(\df x''|z) \\
= & \int_{\X} \bnorm{K^{\ast (n-1)}\big( (x'',z), \cdot \big) - \Pi^{\ast}(\cdot)}
    \Pi_{X|Z}(\df x''|z)   \;.
\end{align*}
Now using the assumed geometric ergodicity of $K^*$, we have
\[
\bnorm{K_Z^n(z,\cdot) - \Pi_Z(\cdot)}
\le \rho_1^{n-1} \int_{\X} M_1(x,z) \Pi_{X|Z}(\df x|z)
= \frac{M'_1(z)}{\rho_1} \rho_1^n \;,
\]
where we have defined $M'_1(z) := \int_{\X} M_1(x,z) \, \Pi_{X|Z}(\df
x|z)$.  Now using the integrability of $M_1(x,z)$ with respect to
$\Pi^*(\df x,\df y,\df z)$, we have
\begin{align*}
\int_{\X} \int_{\Y} \int_{\Z} M_1(x,z) \,
\Pi^*(\df x,\df y,\df z)
& = \int_{\X} \int_{\Z} M_1(x,z) \, \Pi_{XZ}(\df x,\df z) \\
& = \int_{\Z} \bigg[ \int_{\X} M_1(x,z) \, \Pi_{X|Z}(\df x|z) \bigg] \Pi_Z(\df z)  \\
& = \int_{\Z} M'_1(z) \, \Pi_Z(\df z) \\
& < \infty \;.
\end{align*}
It follows immediately that the $Z$-marginal chain is geometrically
ergodic, and that its geometric rate of convergence is no larger than
$\rho_1$.
Again, the $Z$-marginal chain and the block Gibbs chain share
the same geometric rate of convergence.

Now assume that the block Gibbs sampler is geometrically ergodic.
This implies that the $(X,Y)$-marginal chain is also geometrically
ergodic with the same rate, so, for all $n \in \mathbb{N}$ and all
$(x,y) \in \X \times \Y$,
\[
\bnorm{K_{XY}^n \big((x,y),\cdot) - \Pi_{XY}(\cdot)}
\le M_2(x,y) \rho_2^n
\]
where $\rho_2 \in [0,1)$ and $\int_\X \int_\Y M_2(x,y) \Pi(\df x,\df y) < \infty$.
It's easy to see that
\begin{align*}
\Pi^*(\df y',\df z',\df x')
& = \Pi_{X|Z}(\df x'|z') \Pi_{YZ}(\df y',\df z') \\
& = \int_{\Y} \int_{\X}
    \Pi_{X|Z}(\df x'|z') \Pi_{Z|XY}(\df z'|x'',y') \Pi_{XY}(\df x'',\df y')
    \Pi_{Y|XZ}(\df y''|x,z)  \;.
\end{align*}
Combining this with \eqref{eq:2}, we have
\begin{align*}
 & \bnorm{K^{*n}\big( (y,z,x), \cdot \big) - \Pi^*(\cdot)} \\
 = & \sup_{0\le f \le 1} \bigg| \int_{\X} \int_{\Y} \int_{\Z} f(x',y',z')
     \int_{\Y} \int_{\X} \Pi_{X|Z}(\df x'|z') \Pi_{Z|XY}(\df z'|x'',y') \\
    & \hspace*{65mm} K_{XY}^{n-1}\big( (x, y''), (\df x'',\df y') \big)
     \Pi_{Y|XZ}(\df y''|x,z)\\
    & \hspace*{10mm} - \int_{\X} \int_{\Y} \int_{\Z} f(x',y',z')
     \int_{\Y} \int_{\X} \Pi_{X|Z}(\df x'|z') \Pi_{Z|XY}(\df z'|x'',y')
     \Pi_{XY}(\df x'',\df y')
     \Pi_{Y|XZ}(\df y''|x,z) \bigg| \\
= & \sup_{0\le f \le 1} \bigg|
     \int_{\Y} \int_{\X} \int_{\Y}
     \bigg[ \int_{\Z} \int_{\X} f(x',y',z') \Pi_{X|Z}(\df x'|z') \Pi_{Z|XY}(\df z'|x'',y') \bigg] \\
    & \hspace*{65mm} K_{XY}^{n-1}\big( (x, y''), (\df x'',\df y') \big)
     \Pi_{Y|XZ}(\df y''|x,z)\\
    & \hspace*{10mm} - \int_{\Y} \int_{\X} \int_{\Y}
     \bigg[ \int_{\Z} \int_{\X} f(x',y',z') \Pi_{X|Z}(\df x'|z') \Pi_{Z|XY}(\df z'|x'',y') \bigg]
     \Pi_{XY}(\df x'',\df y')
     \Pi_{Y|XZ}(\df y''|x,z) \bigg| \\
\le & \sup_{0\le g \le 1} \bigg| \int_{\Y} \int_{\X} \int_{\Y}  g(x'',y')
  \bigg[  K_{XY}^{n-1}\big( (x, y''), (\df x'',\df y') \big)
        - \Pi_{XY}(\df x'',\df y') \bigg]
    \Pi_{Y|XZ}(\df y''|x,z) \bigg| \\
\le & \int_{\Y}
  \bigg\{ \sup_{0\le g \le 1} \bigg|
  \int_{\X} \int_{\Y}  g(x'',y')
  \bigg[  K_{XY}^{n-1}\big( (x, y''), (\df x'',\df y') \big)
        - \Pi_{XY}(\df x'',\df y') \bigg] \bigg| \bigg\}
    \Pi_{Y|XZ}(\df y''|x,z)  \\
= &  \int_{\Y}
  \bnorm{K_{XY}^{n-1}\big((x, y''), \cdot \big) - \Pi_{XY}(\cdot)}
    \Pi_{Y|XZ}(\df y''|x,z) \;.
\end{align*}
Now using the assumed geometric ergodicity of $K_{XY}$, we have
\[
\bnorm{K^{*n}\big( (y,z,x), \cdot) - \Pi^*(\cdot)}
\le \rho_2^{n-1} \int_{\Y} M_2(x,y') \Pi_{Y|XZ}(\df y'|x,z)
= \frac{M'_2(x,z)}{\rho_2} \rho_2^n \;,
\]
where we have defined $M'_2(x,z) := \int_\Y M_2(x,y') \Pi_{Y|XZ}(\df y'|x,z)$.
Using the integrability of $M_2(x,y)$ with respect to
$\Pi_{XY}(\df x,\df y)$, we have
\begin{align*}
\int_\X \int_\Y M_2(x,y) \, \Pi_{XY}(\df x,\df y)
& = \int_\X \int_\Y \int_\Z M_2(x,y) \, \Pi(\df x,\df y,\df z) \\
& = \int_\X \int_\Z
    \bigg[ \int_\Y M_2(x,y) \, \Pi_{Y|XZ}(\df y|x,z) \bigg]
    \Pi_{XZ}(\df x,\df z) \\
& = \int_\X \int_\Z M'_2(x,z) \,
    \Pi_{XZ}(\df x,\df z) \\
& < \infty \;.
\end{align*}
It follows immediately that the out-of-order chain is geometrically
ergodic, and that its geometric rate of convergence is no larger than
$\rho_2$.

\end{proof}

\vspace*{8mm}

\noindent {\bf \large Acknowledgment}.  The authors are grateful to
Qian Qin for helpful conversations.
\bigskip

\bibliographystyle{plainnat}
\bibliography{refs_ooo}

\end{document}